\documentclass{amsart}

\usepackage{amssymb}
\usepackage{mathrsfs}
\usepackage{booktabs}
\usepackage{hyperref}
\usepackage{verbatim}
\usepackage{url}
\newtheorem{theorem}{Theorem}[section]
\newtheorem{thm}[theorem]{Theorem}
\newtheorem{lem}[theorem]{Lemma}

\newtheorem{cor}[theorem]{Corollary}

\theoremstyle{definition}
\newtheorem{defin}[theorem]{Definition}

\theoremstyle{remark}
\newtheorem{rem}[theorem]{Remark}

\numberwithin{equation}{section}

\newcommand{\FF}[1]{\mathbb F_{#1}}
\newcommand{\CF}[1]{\overline{\mathbb F}_{#1}}

\newcommand{\CC}{\mathbb C}

\newcommand{\irr}{\textup{Irr}}
\newcommand{\sbs}{\subseteq}
\newcommand{\nor}{\vartriangleleft}

\newcommand{\End}{\textup{End}}

\newcommand{\id}{\textup{id}}

\newcommand{\ul}{\underline}
\newcommand{\ol}{\overline}
\newcommand{\sym}{\operatorname{Sym}}

\newcommand{\msupp}{\operatorname{MinSupp}}
\newcommand{\supp}{\operatorname{Supp}}
\newcommand{\Fix}{\operatorname{Fix}}
\newcommand{\mr}{\operatorname{mr}}

\newcommand{\Gal}{\operatorname{Gal}}
\newcommand{\soc}{\operatorname{Soc}}
\renewcommand{\phi}{\varphi}

\begin{document}
\title{Random bases for coprime linear groups} 

\author{H\"ulya Duyan} \address{Department of Mathematics, Central
  European University, N\'ador utca 9., H-1051, Budapest, Hungary}
\email{duyan\_hulya@phd.ceu.edu}

\author{Zolt\'an  Halasi} 
\address{ Department of Algebra and Number Theory,
  E\"otv\"os University, P\'azm\'any P\'eter s\'et\'any 1/c, H-1117,
  Budapest, Hungary \and Alfr\'ed R\'enyi Institute of Mathematics,
  Hungarian Academy of Sciences, Re\'altanoda utca 13-15, H-1053,
  Budapest, Hungary\\
  ORCID: \url{https://orcid.org/0000-0002-1305-5380}
} 
\email{zhalasi@cs.elte.hu and halasi.zoltan@renyi.mta.hu}

\author{K\'aroly Podoski}
\address{Alfr\'ed R\'enyi Institute of Mathematics,
  Hungarian Academy of Sciences, Re\'altanoda utca 13-15, H-1053,
  Budapest, Hungary}
\email{podoski.karoly@renyi.mta.hu}

\thanks{The work of the second and third authors on the project leading
  to this application has received funding from the European Research
  Council (ERC) under the European Union's Horizon 2020 research and
  innovation programme (grant agreement No 741420). 
  The second author was also supported by the J\'anos
  Bolyai Research Scholarship of the Hungarian Academy of Sciences 
  and  by the National
  Research, Development and Innovation Office (NKFIH) Grant
  No.~K115799.}
\maketitle
\begin{abstract}
  The minimal base size $b(G)$ for a permutation group $G$, is a
  widely studied topic in the permutation group theory.  Z.~Halasi and
  K.~Podoski \cite{HP} proved that $b(G)\leq 2$ for coprime linear
  groups. Motivated by this result and the probabilistic method used
  by T.~Burness, M.W.~Liebeck and A.~Shalev, it was asked by L.~Pyber
  \cite{pype} that for coprime linear groups $G\leq GL(V)$, whether
  there exists a constant $c$ such that the probability of that a
  random $c$-tuple is a base for $G$ tends to 1 as $|V|\to\infty$.
  While the answer to this question is negative in general, it is
  positive under the additional assumption that $G$ is even primitive
  as a linear group. In this paper, we show that almost all
  $11$-tuples are bases for coprime primitive linear groups.
\end{abstract}

\section{Introduction}
For a finite permutation group $G$, a subset $B$ of $\Omega$ is called
 a base for $G$, if the pointwise stabilizer of $B$ in $G$ is trivial. 
The concept of base plays a fundamental role in the development 
of the algorithms for permutation groups and these algorithms are 
significantly faster if the size of the base is small 
(see the book of \'A.~Seress \cite{sepe}).
The minimal size of a base for $G$ acting on $\Omega$ is denoted by $b(G)$.
L.~Pyber (\cite{pyas}) showed that there exists a universal constant $c>0$
such that almost all subgroups $G$ of $\sym(n)$ satisfy that $b(G) >cn$.

On the other hand, there are several important group classes where the
minimal base size $b(G)$ can be bounded by a fixed constant $c$.
\'A.~Seress \cite{Seress} showed that $b(G) \leq 4$ for a solvable
primitive group $G$.  For an almost simple primitive permutation group
$G$, a famous conjecture of P.~J.~Cameron and W.~M.~Kantor \cite{caka}
states that there exists an absolute constant such that $b(G)\leq c$
for all non-standard primitive permutation groups $G$.  In
\cite{cape}, P.~J.~Cameron suggested that $c$ can be chosen to $6$ apart
from the Mathieu group $M_{24}$ with its natural action, where the
minimal base size is $7$.  The Cameron--Kantor conjecture was proved
by M.~W.~Liebeck and A.~Shalev in \cite{lisa}.  However, this was an
existence result for $c$, using probabilistic method without yielding
any explicit value for this constant.  Finally, T.~C.~Burness, M.~W.~Liebeck
and A.~Shalev \cite{buli} proved that if $G$ is a finite almost simple
group in a primitive faithful non-standard action, then $b(G)\leq 7$,
with equality if and only if G is the Mathieu group $M_{24}$ in its
natural action of degree 24.

Furthermore, they proved that if G is a finite almost simple group,
and $\Omega$ is a primitive faithful non-standard $G$-set then the
probability that a random 6-tuple in $\Omega$ is a base for G tends to
1 as $|\Omega|\to\infty$. For a finite vector space $V$, a linear
group $G\leq GL(V)$ is called coprime, if $(|G|,|V|)=1$.  D.~Gluck and
K.~Magaard \cite{glma} proved that for such a group the minimal base
size of $G$ is bounded by an absolute constant that is $b(G)\leq 94$.
Z.~Halasi and K.~Podoski \cite{HP} improved this result by showing
that $b(G)\leq 2$ and this estimation is sharp.  Based on this result
and the random base result of T.~C.~Burness, M.~W.~Liebeck and
A.~Shalev, L.~Pyber \cite{pype} asked whether for a coprime linear
group $G\leq GL(V)$ there exists an absolute constant $c$ such that a
random $c$-tuple in $V$ is a base for $G$ tends to 1 as
$|V|\to\infty$.  We answer this question affirmatively by showing that
\begin{thm}
  Let $V$ be a finite vector space and 
  $G\leq GL(V)$ be a coprime primitive linear group, i.e. $(|G|,|V|)=1$.
  Then the probability that a random $11$-tuple in $V$ is a base for
  $G$ tends to 1 as $|V|\to\infty$.
 \end{thm} 
 In fact, we also give lower bounds for this probability in terms of
 the base field and the dimension of $V$. Our bounds highly depends on
 the structure of $G$. As a part of our argument, we give a general
 structure theorem for maximal coprime primitive linear groups in Theorem
 \ref{thm:generalcase}.  

 For any positive integer $c$ let us define the probability
 \[
 Pb(c,G,V):=P(\textrm{random }v_1,\ldots,v_c\in V\textrm{ is a base for }G).
 \]
 The main goal of this paper is to prove the following 
\begin{thm}\label{thm:randombase}
  Let $V$ be an $n$-dimensional vector space over the finite field
  $\FF q$ and let $G\leq GL(V)$ be a coprime primitive linear group.
  Then for any $c\geq 11$, the probability $Pb(c,G,V)$ is close to zero
  if $|V|$ is large enough.  More precisely, one of the following
  holds.
  \begin{enumerate}
    \item $Pb(c,G,V)\geq 1-\frac{3}{q^{(\frac{c}2-5)\sqrt{n}}}$;
    \item There is an $\FF q^k$ vector space 
      structure on $V$ for some field extension 
      $\FF q^k\geq \FF q$ (possibly $k=1$)
      and a tensor product decomposition $V=V_1\otimes_{\FF q^k} U$ over 
      $\FF q^k$ with $1\leq \dim(U)< \dim(V_1)\leq n/k$ such that 
      $G\leq \Gamma L(\FF q^k,n/k)$ and $H=G\cap GL((\FF q^k,n/k)$
      preserves this tensor product decomposition. 
      Furthermore, $H=H_1\otimes H_2$ with $H_1\leq GL(V_1)$, $H_2\leq GL(U)$
      are absolutely irreducible linear groups, 
      and $S_1=\soc(H_1/Z(H_1))$ is a non-Abelian simple group. 
      \begin{enumerate}
        \item If $S_1$ is not an alternating group, then
          \[Pb(c,G,V)\geq 1-\Big(\frac{1}{q^{(c-4)\sqrt{\dim(V)}}}+
          \frac{2}{|V|^{(c-10)/80}}\Big);\]
        \item If $S_1\simeq A_m$ for some $m$ and 
          $V_1$ is not an irreducible component of the natural permutation 
          $\FF q^k A_m$-module, then 
          \[Pb(c,G,V)\geq 1-\frac{3}{q^{\frac{c-10}{16}\sqrt{\dim(V)}}};\]
        \item If $S_1\simeq A_m$ for some $m$ and $V_1$ is the
          non-trivial irreducible component of the natural permutation
          $\FF q^k A_m$-module, then
          \[Pb(c,G,V)\geq 1-\frac{3}{n^{c-2}}.\]
      \end{enumerate}
  \end{enumerate}
\end{thm}
\newpage
\begin{rem}\leavevmode
  \begin{enumerate}
  \item Let $Z=Z(GL(V))\simeq \FF q^\times$ denote the group of scalar
    transformations on $V$. If $G\leq GL(V)$ is a coprime linear group
    on $V$, then so is $GZ\geq G$ and we have $Pb(c,G,V)\geq
    Pb(c,GZ,V)$. Therefore, for the rest of this paper we will always assume
    that $G$ contains $Z$.
  \item The assumption ``primitive'' is necessary here. To see this,
    let $H\leq GL(n,q)$ be the group of all invertible diagonal
    matrices, so $H\simeq(\FF q^\times)^n$. Then $v_1,\ldots,v_c\in
    \FF q^n$ is a base for $H$ if and only if for each $1\leq i\leq n$
    the $i$-th component of some $v_j$ is non-zero. For any fixed $i$,
    this has probability $(1-1/q^c)$, so we have
    \[
    Pb(c,H,\FF q^n)
    =\Big(1-\frac{1}{q^c}\Big)^n,
    \]
    which is close to zero for any fixed $c$ and big enough $n$.  If
    $(q,n)=1$, then one can add the regular permutation action of
    $C_n$ on the components of $\FF q^n$ to get the coprime
    irreducible linear group $G=H\rtimes C_n\leq GL(n,q)$ satisfying
    $\lim_{n\to\infty}Pb(c,G,\FF q^n)=0$.
  \end{enumerate}
\end{rem}
\section{Bounds on $Pb(c,G,V)$ in terms of supports and character ratios}
In order to prove Theorem \ref{thm:randombase}, our primary tool will be 
the concept of support for elements of a linear group.
\begin{defin}\label{def:minsupp}
  For a linear group $G\leq GL(V)$ and a $g\in G$ the
  \emph{fixed-point space} and the \emph{support} of $g$ are defined as
  \[
  \Fix(g):=\{v\in V\,|\,g(v)=v\} \textrm{\quad and \quad}
  \supp(g):=\dim (V)-\dim (\Fix_V(g)).
  \]
  Furthermore, let the \emph{minimum support} of $G$ be defined as 
  \[\msupp(G):=  \min_{1\neq g\in G}\supp(g).
  \]
  We use the notation $\Fix_V(g),\supp_V(g)$ and $\msupp_V(G)$ if 
  we also want to highlight the vector space on which the group acts. 
\end{defin}
\begin{rem}
  If $G$ strictly contains $Z$, then $\msupp(G)$ equals
  \[
  \min_{g\in G\setminus Z}\big(\dim (V)-\max_{\lambda\in \FF q^\times}
  (\dim(\ker(g-\lambda\cdot \id_V)))\big).
  \]
\end{rem}
In order to get bounds for $\msupp_V(G)$ in case of $G\leq GL(V)$ is a
quasisimple coprime linear group, we will use results from character
ratios of complex irreducible characters of such groups.
\begin{defin}
  For a finite group $G$ and $\chi\in\irr(G)$ with $\chi(1)\neq 1$
  let us define the \emph{maximal character ratios}
  \[
  \mr(G,\chi):=\max_{g\notin Z(\chi)}\frac{|\chi(g)|}{\chi(1)}
  \textrm{\quad and\quad}
  \mr(G):=\max_{\chi\in\irr(G),\,\chi(1)\neq 1} \mr(G,\chi). 
  \]
  Clearly, $\mr(G)<1$ for every finite group $G$.
\end{defin}
The connection between minimal support and maximal character ratio is
described in the following Lemma.
\begin{lem}\label{lem:msupp-mr_connection}
  Let $V$ be an $n$-dimensional vector space over the finite field $\FF q$ and 
  let $G\leq GL(V)$ be a non-Abelian coprime irreducible linear group. 
  Then we have 
  \[
  \msupp_V(G)\geq \frac{\dim(V)}2\Big(1-\mr(G)\Big).
  \]

  Moreover, if  $\chi\in\irr(G)$ is 
  any irreducible component of the Brauer character associated to $V$, then 
  \[\msupp_V(G)\geq \frac{1}{2}\Big(\chi(1)-\max_{g\notin Z(\chi)}|\chi(g)|\Big).\]
\end{lem}
\begin{proof}
  Let $\CF q$ be the algebraic closure of $\FF q$ and let
  $\ol{V}=V\otimes \CF q$ be the $\CF q G$-module arising from the
  $\FF q G$-module $V$. Let $\ol{V}=V_1\oplus\ldots\oplus V_t$ be the
  decomposition of $\ol V$ into irreducible $\CF q G$-modules. Then
  the corresponding representations $G\mapsto GL(V_i)$ form a single
  Galois conjugacy class by \cite[Theorem 9.21]{Ibook}, so
  $\supp_{V_i}(g)=\frac{1}{t}\supp_V(g)$ holds for every $g\in G$.
  Let $\chi_i:G\mapsto \CC$ be the irreducible Brauer character
  associated to $V_i$ for each $1\leq i\leq t$. Since $(q,|G|)=1$, we
  get $\chi_i\in \irr(G)$ by \cite[Theorem 15.13]{Ibook}. Furthermore, 
  \[
  \chi_i(1)=\dim(V_i)\textrm{\quad and \quad}
  [\chi_i{\langle g\rangle},1_{\langle g\rangle}]=\dim(\Fix_{V_i}(g)). 
  \]  

  For any $g\in G$ we have $\chi_1(g)=(\chi_1(1)-k)\cdot
  1+\varepsilon_1+\ldots+\varepsilon_k$ where $k=\supp_{V_1}(g)$ and
  $\varepsilon_1,\ldots,\varepsilon_k$ are $o(g)$-th root of unity.
  Then $|\chi_1(g)|\geq \chi_1(1)-2k=\chi_1(1)-2\supp_{V_1}(g)$ holds,
  so $2\msupp_{V_1}(G)\geq \chi_1(1)-\max_{g\notin
    Z(\chi_1)}|\chi_1(g)|$.  (Note that the assumption that $G$ is
  non-Abelian implies that the $\chi_i$ are non-linear
  characters. Furthemore, if $1\neq g\in Z(\chi_1)$, then
  $\supp_{V}(g)=\dim(V)$, so $\msupp_V(G)=\min_{g\notin
    Z(\chi_1)}\supp_V(g)$ must hold.)

  It follows that
  \begin{align*}
    2\msupp_V(G)&=
    2t\msupp_{V_1}(G)\geq t(\chi_1(1)-\max_{g\notin Z(\chi_1)}|\chi_1(g)|)\\
    &=t\chi_1(1)(1-\mr(G,\chi_1))\geq \dim(V)(1-\mr(G)).
  \end{align*}

  Now, the first inequality proves the second claim, while the second
  inequality proves the first claim.
\end{proof}
\begin{lem}\label{lem:P-Msupp}
  \[
  Pb(c,G,V) \geq 1-\sum_{1\neq g\in G}\frac{1}{q^{c\cdot\supp(g)}}
  \geq 1-\frac{|G|}{q^{c\cdot \msupp(G)}} \geq
  1-\frac{1}{|V|^{c(1-\mr(G))/2-2}}.
  \]
  In particular, $Pb(c,G,V)\geq 1-\frac{1}{|V|^\varepsilon}$
  for $c\geq \frac{4+2\varepsilon}{1-\mr(G)}$.
\end{lem}
\begin{proof}
  \begin{multline*}
    P(\{v_1,\ldots, v_c\}\sbs V\textrm{ is not a base for }G)
    \leq \sum_{1\neq g\in G}P(g(v_i)=v_i,\;\forall \; 1\leq i\leq c)\\
    =\sum_{1\neq g\in G}\left(\frac{|\Fix(g)|}{|V|}\right)^c
    =\sum_{1\neq g\in G}\frac{1}{q^{c\cdot\supp(g)}}
    \leq \frac{|G|}{q^{c\cdot \msupp(G)}}\\
    \leq \frac{|V|^2}{(q^n)^{c(1-\mr(G))/2}}=\frac{1}{|V|^{c(1-\mr(G))/2-2}},
  \end{multline*}
  and the claim follows.
\end{proof}
\section{Bounds for character ratios and for 
  minimal supports of quasisimple linear groups}
The goal of this section is to give lower bounds for minimal supports
of coprime quasisimple groups $G\leq GL(V)$ in terms of $|G|$ and
$\dim(V)$.  

First we handle the case when $G$ is a sporadic group or a
finite quasisimple group of Lie type.  For such groups, we use bounds
for their maximal character ratios $\mr(G)$.
\begin{thm}\label{thm:bound_to_mr}
  Let $G$ be a finite quasisimple group such that $G/Z(G)$ is not an
  alternating group.
  \begin{enumerate}
  \item If $G/Z(G)$ is a sporadic simple group, then $\mr(G)<0.54$. 
  \item If $G=G(r)$ is a finite quasisimple group of Lie type over
    the field $\FF r$, then
    \[
    \mr(G)\leq
    \left\{
      \begin{array}{ll}
        \max\Big(\frac{1}{\sqrt r-1},\frac{9}{r}\Big)&\textrm{ if }r>9;\\
        \frac{19}{20}&\textrm{ if }r\leq 9.
      \end{array}
    \right.
    \]
  \end{enumerate}
\end{thm}
\begin{proof}
  We checked part (1) for the covering groups of the sporadic simple
  groups by using the GAP \cite{GAP} Character table library and also the
  undeposited GAP package FUtil to turn cyclotomic complex numbers
  into floating ones in order to be able to compare the values of
  $|\chi(g)|$ for various $g$ and $\chi$. 
  
  Regarding part (2), it is a simplified version of a result of Gluck 
  \cite{G}. (For a summary of his results, see also \cite[Theorem
  2.4]{Liebeck}). 
\end{proof}
\begin{rem}
  For simple groups of alternating type there is no general upper
  bound for $\mr(G)$ smaller than $1$. Moreover it can be shown that
  for every $\varepsilon>0$, the number of irreducible characters
  $\chi\in\irr(S_m)$ (or $\chi\in\irr(A_m)$) satisfying
  $\mr(S_m,\chi)>1-\varepsilon$ is not bounded if $m$ is large enough. 
\end{rem}
\begin{cor}\label{cor:msupp-dim_bound}
  Let $V$ be a vector space over the finite field $\FF q$ and let
  $G=Z\cdot G_0\leq GL(V)$ where $G_0$ is a coprime quasisimple
  irreducible linear group which is not of alternating type. 
  Then $\msupp_V(G)\geq \frac{1}{40}\dim(V)$.
\end{cor}
\begin{proof}
  By Theorem \ref{thm:bound_to_mr}, we have $\mr(G)=\mr(G_0)\leq
  \frac{19}{20}$, so the claim follows from Lemma
  \ref{lem:msupp-mr_connection}.
\end{proof}
Now, we handle the case when $\soc(G/Z(G))$ is an alternating group. 
\begin{thm}\label{thm:threecycle}
  Let $G=S_m$ and $\chi=\chi^{(\lambda)}\in\irr(G)$ corresponding
  to the partition $\lambda=(\lambda_1\geq \ldots\geq\lambda_k)$ of
  $[m]$. Then $\chi^\lambda(1)-\chi^\lambda((123))\geq
  \frac{1}{m-1}\chi^\lambda(1)$ unless
  \[\lambda\in \{(m);\,(1,\ldots,1)\}.\]
\end{thm}
\begin{proof}
  First, we introduce some notation.  Let $\lambda=(\lambda_1\geq
  \ldots\geq\lambda_k)$ be a partition of $m$, different from the two
  exceptional ones given in the theorem.  For any natural numbers
  $i_1,\ldots,i_k$ let $\chi^{\lambda-\{i_1,\ldots,i_k\}}$ be the
  character of $S_{m-k}$ corresponding to the Young diagram obtained
  from the diagram of $\lambda$ by deleting the last cells of the
  $i_1$-th,\ldots,$i_k$-th row in that order with the assumption that
  $\lambda-\{i_1,\ldots,i_s\}$ is a valid Young diagram for each
  $1\leq s\leq k$. Otherwise, we define
  $\chi^{\lambda-\{i_1,\ldots,i_k\}}$ as the constant zero function on
  $S_{m-k}$.

  By the Murnaghan-Nakayama rule (see \cite[21.1]{James}),
  \begin{align*}
  \chi^\lambda((123))&\leq\sum_{\nu\in\{\lambda-rh(3)\}}\chi^{\nu}(1)
  +\sum_{\nu\in\{\lambda-rh(1,1,1)\}}\chi^{\nu}(1)\\
  &=\sum_{\nu\in\{\lambda-rh(3)\}}\chi^{\nu}(1)+
  \sum_{\nu\in\{\overline{\lambda}-rh(3)\}}\chi^{\nu}(1)
  \end{align*}
  where $\{\lambda-rh(*)\}$ denotes the set of partitions of $m-3$
  which we can get from the Young-diagram of $\lambda$ be removing a
  rim $3$-hook of type $(*)$ such that the remaining cells form a
  valid Young diagram.

  On the other hand, by using the branching rule (three times) one
  gets
  \[
  \chi^\lambda(1)=\sum_{i,j,k}\chi^{\lambda-\{i,j,k\}}(1).
  \]

  Let $\nu\in\{\lambda-rh(3)\}$. Then $\nu=\lambda-\{i,i,i\}$ for some
  (unique) $i$. Now, there is a $j\neq i$ such that
  $\tau=\lambda-\{i,i,j\}$ is a valid Young diagram. Then 
  both induced characters $(\chi^\tau)^{S_{m-2}}$ and $(\chi^\nu)^{S_{m-2}}$
  contain $\chi^{\lambda-\{i,i\}}$ as a component which results 
  $\chi^{\nu}(1)\leq \chi^{\lambda-\{i,i\}}(1)\leq (m-2)\chi^{\tau}(1)$.
  The same argument can be applied to any $\nu\in\{\overline{\lambda}-rh(3)\}$.
  It follows that 
  \begin{align*}
  \chi^\lambda(1)&\geq \sum_i\chi^{\lambda-\{i,i,i\}}(1)\Big(1+\frac{1}{m-2}\Big)+
  \sum_i\chi^{\overline{\lambda}-\{i,i,i\}}(1)\Big(1+\frac{1}{m-2}\Big)\\
  &\geq \frac{m-1}{m-2} \chi^\lambda((123)).
  \end{align*}
  Hence $\chi^\lambda(1)-\chi^\lambda((123))\geq \frac{1}{m-1} \chi^\lambda(1)$
  which proves the claim.
\end{proof}
This result will be adequate for our purposes only if the degree of
$\chi$ is large enough.  In order to get an overall picture about the
form of Young diagrams defining characters of small degree, we will
use a result of Rasala \cite{Rasala}. In what follows, we use the
terminology from Rasala's paper. For any partition $\lambda$ of $m$,
let $|\lambda|=m$ be the order of $\lambda$ and let $\lambda^*$ be the
partition dual to $\lambda$.  The partition $\lambda$ is called
primary, if $\lambda\geq \lambda^*$, where $\geq$ denotes the standard
ordering on partitions.  If
$\lambda=(\lambda_1\geq\ldots\geq\lambda_k)$ is a partition of $k$ and
$m\geq \lambda_1+k$, then let $m/\lambda$ denote the partition of $m$
defined as $m/\lambda=(m-k\geq \lambda_1\geq\ldots\geq\lambda_k)$ and
let $\phi_\lambda(m):=\chi^{m/\lambda}(1)$ be the degree of the
character of $S_m$ associated to $m/\lambda$. (Note that
$\phi_\lambda(m)$ is a polynomial in $m$ by \cite[Theorem A]{Rasala}.)
For any set $P$ of partitions of $k$ and for $m$ large enough, let
$L(P,m):=\{\phi_\lambda(m)\,|\,\lambda\in P\}$ and let $\delta(P,m)$
be the largest degree in $L(P,m)$. Then $P$ is said to be $m$-minimal,
if for every primary partition $\mu$ of $m$ either $\chi^\mu(1)>\delta(L,P)$ or 
$\mu=m/\lambda$ for some $\lambda\in P$. 

By \cite[Main Theorem 1.]{Rasala} (for $k=3$) we have
\begin{thm}
  Let $P_3$ be the set of all partitions of order at most $3$, that
  is, $P_3=\{\emptyset;(1);(2);(1,1);(3);(2,1);(1,1,1)\}$. Then $P_3$
  is $m$-minimal for every $m\geq 15$.
\end{thm}
Thus, by using the hook length formula and the Murnaghan-Nakayama rule
we can calculate the exact values of $\chi^\lambda(1)$ and
$\chi^\lambda((123))$ when $\chi^\lambda(1)$ is among the first seven
smallest character degrees of $S_m$ for $m\geq 15$.  Otherwise, we get a
reasonably large lower bound for $\chi^\lambda(1)$.  (Note that
$\lambda$ or $\lambda^*$ is primary and
$\chi^\lambda(1)=\chi^{\lambda^*}(1),\ \chi^\lambda((123))=\chi^{\lambda^*}((123))$
holds for every partition $\lambda$ of $m$.)
\begin{cor}
  Let $\lambda$ be a partition of $m$ for $m\geq 15$ and let
  $\chi^\lambda\in\irr(S_m)$ be the character of $S_m$ associated to $\lambda$. 
  Then $\chi^\lambda(1)$ and $\chi^\lambda((123))$ are as given in 
  Table \ref{tab:smalldegreechi}
  \renewcommand{\arraystretch}{1.5} 
  \begin{table}[ht]
    \begin{tabular}{@{}lll@{}}
      \toprule
      $\lambda$ or $\lambda^*$&$\chi^\lambda(1)=\chi^{\lambda^*}(1)$&
      $\chi^\lambda((123))=\chi^{\lambda^*}((123))$\\
      \midrule
      $(m)        $&$1$                           &$1$\\
      $(m-1,1)    $&$m-1$                         &$m-4$\\
      $(m-2,2)    $&$\frac{1}{2}m(m-3)$           &$\frac{1}{2}(m-3)(m-6)$\\
      $(m-2,1,1)  $&$\frac{1}{2}(m-1)(m-2)$       &$\frac{1}{2}(m-4)(m-5)$\\
      $(m-3,3)    $&$\frac{1}{6}m(m-1)(m-5)$      &$\frac{1}{6}(m-3)(m-4)(m-8)+1$\\
      $(m-3,2,1)  $&$\frac{1}{3}m(m-2)(m-4)$      &$\frac{1}{3}(m-3)(m-5)(m-7)-1$\\
      $(m-3,1,1,1)$&$\frac{1}{6}(m-1)(m-2)(m-3)$  &$\frac{1}{6}(m-4)(m-5)(m-6)+1$\\
      \bottomrule
    \end{tabular}
    \caption{Character values of $S_m$ when the degree is small.}\label{tab:smalldegreechi}
  \end{table}
  or  $\chi^\lambda(1)>\frac{1}{3}m(m-2)(m-4)$.
\end{cor}
Now, we give an analogue of Corollary \ref{cor:msupp-dim_bound} for 
alternating-type groups.
\begin{cor}\label{cor:msupp-dim_symmetric}
  Let $V$ be a vector space over the finite field $\FF q$ and let
  $G=Z\cdot G_0\leq GL(V)$ where $G_0$ is a coprime irreducible linear
  group and $G_0/Z(G_0)\simeq A_m$ for some $m\geq 5$.  Let us assume
  that $V$ is not a component of the natural permutation $\FF q
  A_m$-module. Then $\msupp_V(G)\geq \frac{1}{16}\sqrt{\dim(V)}$.
\end{cor}
\begin{proof}
  As in the proof of Lemma \ref{lem:msupp-mr_connection},
  $\msupp_V(G)=t\cdot\msupp_{V_1}(G)$ and $\dim(V)=t\cdot\dim(V_1)$
  where $V_1$ is an (absolutely) irreducible component of $\CF q
  G$-module $V\otimes \CF q$. Then the claim clearly follows if we
  prove that $\msupp_{V_1}(G)\geq \frac{1}{16}\sqrt{\dim(V_1)}$.  In other words,
  we can assume that $V$ is absolutely irreducible.  First let us
  assume that $G_0\simeq A_m$ for some $m\geq 9$.  Let
  $\phi\in\irr(A_m)$ be the Brauer character associated to $V$ and
  $\chi\in\irr(S_m)$ above $\phi$, i.e. $[\chi_{A_m},\phi]\neq
  0$. Then either $\chi_{A_m}=\phi$ (if $\chi$ is not self-dual) or
  $\chi_{A_m}=\phi+\phi^{(12)}$ (if $\chi$ is self-dual). In the
  latter case $\phi((123))=\chi((123))/2$, since the conjugacy class
  $(123)^{S_m}$ does not split in $A_m$. Let $\epsilon$ be $1$ or
  $1/2$ according to these cases, so $\phi(1)=\epsilon \chi(1)$ and
  $\phi((123))=\epsilon\chi((123))$.  By \cite[Result 2.]{Rasala}, we
  have $\dim(V)=\epsilon\chi(1)\geq \frac{1}{2}m(m-3)$.  If
  $\phi((123))<0$, then $\supp_V((123))\geq \frac{1}2\dim(V)\geq
  \frac{1}4\sqrt{\dim(V)}$ holds trivially. Otherwise, by using Lemma
  \ref{lem:msupp-mr_connection} and Theorem \ref{thm:threecycle} we
  get that
  \begin{align*}
  \supp_V((123))&\geq \frac{1}{2}(\phi(1)-|\phi((123))|)
  =\frac{\epsilon}2 (\chi(1)-\chi((123)))=\frac{\epsilon\chi(1)}{2(m-1)}\\
  &=\frac{\dim(V)}{2(m-1)}\geq \frac{\sqrt{m(m-3)/2}\sqrt{\dim(V)}}{2(m-1)}
  \geq \frac{1}{4}\sqrt{\dim(V)}.
  \end{align*}
  For any
  element $1\neq g\in A_m$ there are $x,y\in A_m$ such that $[g,x,y]$
  is a three-cycle.  Applying Lemma \ref{lem:commutator_supp} twice,
  we get that $\supp_V(g)\geq \frac{1}4\supp_V((123))\geq 
  \frac{1}{16}\sqrt{\dim(V)}$.

  Now, let us assume that $m>7$ and $G_0$ is the universal covering
  group of $A_m$, so $G_0\simeq 2.A_m$. Let $z\in G_0$ be the
  generator of $Z(G_0)\simeq C_2$ and let $\bar{g}\in A_m$ denote the
  image of any $g\in G_0$ under the natural surjection by $G_0\mapsto
  A_m$.  Then $z$ acts on $V$ as a scalar transformation $z(v)=-v$ for
  all $v\in V$, so $\supp_V(z)=\dim(V)$. Let $t\in G_0$ such that
  $\bar{t}=(12)(34)$.  By Theorem \cite[Theorem
  3.9]{HoffmanHumphreys}, $t$ and $tz$ are conjugate, so $z=[h,t]$ for
  some $h\in G_0$. It follows that $\supp_V(t)\geq
  \frac{1}{2}\supp_V(z)=\frac{\dim(V)}2$ by Lemma
  \ref{lem:commutator_supp}. (In fact, by using this argument to $tz$
  instead of $t$ one can prove equality here.) Now, for any $g\in
  G_0\setminus Z$ one can choose $x,y\in G$ such that
  $[\bar{g},\bar{x},\bar{y}]$ is conjugate to $\bar{t}$. Using again
  Lemma \ref{lem:commutator_supp} twice, we get that $\supp_V(g)\geq
  \frac{1}4\supp_V(t)=\frac{1}{8}\dim(V)\geq  \frac{1}{16}\sqrt{\dim(V)}$.

  For the remaining cases, $\dim(V)\leq \sqrt{|G_0|}<16^2$, so
  $\frac{1}{16}\sqrt{\dim(V)}<1\leq \msupp_V(G)$ follows. 
\end{proof}
The next result gives a bound to the order of most coprime quasisimple 
linear groups similar to that of $|G|\leq |V|^2=q^{2\dim(V)}$  
but using the minimal support $\msupp_V(G)$ instead of $\dim(V)$. 
\begin{thm}\label{thm:quasisimple}
  Let $V$ be a vector space over the finite field $\FF q$ and let
  $G=Z\cdot G_0\leq GL(V)$ where $G_0$ is a coprime quasisimple
  irreducible linear group.

  Then one of the following holds:
  \begin{enumerate}
  \item $\log_q|G|\leq d\cdot \msupp_V(G)$ with $d=5$.
  \item $G_0\simeq A_m$ and $V$ is the non-trivial irreducible component
    of the natural permutation module of $A_m$ over $\FF q$.
  \item $G_0=G_0(r)$ is a finite quasisimple group of Lie type
    over the finite field $\FF r$ with $r\leq
    43$, and $|V|$ is bounded by an absolute constant.
  \end{enumerate}
\end{thm}
\begin{proof}
  For any sporadic group $S$, let $\widehat{S}$ be its universal
  covering group and let $q(S)$ be the smallest prime not dividing the
  order of $S$.  By using GAP \cite{GAP}, we checked that for
  every $\chi\in\irr(\widehat{S})$, the inequlity
  \[\log_{q(S)}|\widehat{S}|<
  d\cdot (\chi(1)-\max_{g\in \widehat{S}-Z(\chi)}|\chi(g)|)/2\] holds
  with $d>4.22$. (The largest value is attained for $2.J_2$.)  Now, if
  $G\leq GL(V)$ is any finite quasisimple group with sporadic simple
  quotient $S=G/Z(G)$, then $G$ is a homomorphic image of
  $\widehat{S}$, and we can view $V$ as an irreducible $\FF q
  \widehat{S}$-module (where $q\geq q(S)$). Now, if
  $\chi\in\irr(\widehat{S})$ is any irreducible component of the
  Brauer character corresponding to $V\otimes \CF q$, then
  \begin{align*}
    \log_q|G|&\leq \log_{q(S)}|\widehat{S}|<d\cdot 
    (\chi(1)-\max_{g\in \widehat{S}-Z(\chi)}|\chi(g)|)/2\\
    &\leq d\cdot \msupp_V(\widehat{S})\leq d\cdot \msupp_V(G)
  \end{align*}
  also holds with $d>4.22$ by Lemma \ref{lem:msupp-mr_connection}.

  Next, let $G\simeq A_m$ for some $m\geq 15$. Then we have $m<q$ by
  the coprime assumption.  Let us assume that $V$ is not a component
  of the natural permutation $\FF q A_m$-module. Let
  $\phi\in\irr(A_m)$ be an irreducible component of the Brauer
  character associated to $V$ and $\chi\in\irr(S_m)$ above $\phi$. By
  the proof of Corollary \ref{cor:msupp-dim_symmetric}, we have
  $\phi(1)=\varepsilon\chi(1)$, and $\phi((123))=\varepsilon\chi((123))$, where 
  $\varepsilon$ is $1/2$ or $1$ if $\chi$ is self-dual or not. 

  If $\chi$ is one of the characters given in Table
  \ref{tab:smalldegreechi}, then $\chi$ is not self-dual. In that case
  we have
  \[
  \supp_V((123))\geq \frac{1}2(\chi(1)-|\chi(123)|)\geq \frac{3}{2}(m-3)
  \]
  by using Lemma \ref{lem:msupp-mr_connection} and the last five rows
  of Table \ref{tab:smalldegreechi}.  Otherwise,
  $\chi(1)>\frac{1}{3}m(m-2)(m-4)$, so we have
  \begin{align*}
    \supp_V((123))&\geq \frac{1}2(\phi(1)-|\phi(123)|)\geq 
                    \frac{1}4(\chi(1)-|\chi(123)|)\\
            &\geq \frac{\chi(1)}{4(m-1)}>\frac{m(m-2)(m-4)}{12(m-1)}\geq m-3
  \end{align*}
  holds if $\phi(123)\geq 0$. However, if $\phi(123)<0$, then
  $\supp_V((123))\geq \frac{1}{2}\dim(V)\geq m-3$ holds trivially.
  Thus, $\supp_V((123))\geq m-3$ holds in any case.  Now, for any
  element $1\neq g\in A_m$ there are $x,y\in A_m$ such that $[g,x,y]$
  is a three-cycle.  Applying Lemma \ref{lem:commutator_supp} twice,
  we get that $\supp_V(g)\geq \frac{1}4\supp_V((123))\geq
  \frac{m-3}{4}$ holds for any $1\neq g\in A_m$.  Thus,
  $d\cdot\msupp_V(G)\geq \frac{d(m-3)}{4}\geq m\geq \log_m(m!)\geq
  \log_q|G|$ holds for $d\geq 5$.  

  Now, let $m\geq 12$ and let $G_0$ be the universal covering group of
  $A_m$, so $G_0\simeq 2.A_m$. By the proof of Corollary
  \ref{cor:msupp-dim_symmetric}, we have $\msupp_V(G)\geq
  \frac{1}{8}\dim(V)$. Using \cite[Main Theorem]{KleshchevTiep} we get that
  \begin{align*}
    d\cdot \msupp_V(G)&\geq \frac{d}{8}\dim(V)\geq \frac{d}{8}
  \min\{\chi(1)\,|\,\chi\in\irr(G),\,\chi(z)\neq \chi(1)\}\\
  &\geq d \cdot 2^{\lfloor m/2\rfloor-4}\geq m\geq \log_q|G|
  \end{align*}
  holds for $d\geq 3.25$.  For the remaining members of Alternating
  groups and their covers (i.e for $A_m,\ 12\leq m\leq 14)$, for
  $2.A_m (m=5\textrm{ or }8\leq m\leq 11)$ and for $6.A_6, 6.A_7$ we
  used the same algorithm as for sporadic groups. 
  
  Finally, let $G_0=G_0(r)$ be a quasisimple group of Lie type over 
  a finite field $\FF r$ with $(r,q)=1$. 
  First, suppose that $r\geq 47$. By part (2) of Theorem
  \ref{thm:bound_to_mr}, we have
  \[
  \mr (G)\leq\max\Big(\frac{1}{\sqrt{r}-1},\frac{9}{r}\Big)<\frac{1}{5}.
  \]
  By using \cite[Theorem 1.]{P3-Pyber} and Lemma  \ref{lem:msupp-mr_connection},
  \[\log_q |G|\leq 2n = 5\cdot\frac{2n}5\leq 5\cdot\msupp_V(G)\]
  For the rest of the proof, suppose that $r\leq 43$. Since $\chi
  (1)-2\supp (g)\leq |\chi(g)|$ for any $\chi\in\mathrm{Irr}(G)$, we
  have that
  \[\frac{1}{2}\chi(1)\Big( 1-\frac{|\chi(g)|}{\chi(1)}\Big)\leq\supp(g).\]
  Using that $\mr (G)\leq\frac{19}{20}$ also holds for all quasisimple groups of
  Lie-type by part (2) of Theorem \ref{thm:bound_to_mr}, we obtain that
  \[\frac{\dim(V)}{8}\leq 5 \msupp_V(G)\]
  by using Lemma \ref{lem:msupp-mr_connection} again.  Since (see
  \cite[Table 5.3.A]{Kleidman}) $\dim(V)\geq r^{\mathcal{O}(m)}$
  (where $m$ denotes the rank of $G_0(r)$) and
  $\log_q|G|=\mathcal{O}(m^2\log r)$, there exist only finitely many
  possible pairs $(m,r)$ such that
  $\log_q|G|>5\msupp_V(G)$. Furthermore, for any fixed $(m,r)$, the
  inequality $\log_q|G|\leq 5\msupp_V(G)$ still holds provided that
  $|V|$ is large enough.
\end{proof}
We close this section by handling the case (2) in Theorem
\ref{thm:quasisimple}. In this case $\msupp_V(G)$ is bounded.
(It is $1$ and $2$ for $G_0\simeq S_m$ and $G_0\simeq A_m$, respectively.)
Therefore, we give a direct proof for Theorem \ref{thm:randombase} in this case.
\begin{thm}\label{thm:sym_perm}
  Let $U$ be an $m$-dimensional vector space over $\FF q$,  
  and let $G=S_m$ with its natural permutation action on $U$. 
  Assuming that $(|G|,|U|)=1$, we have 
  \[P(\textrm{random }\ul u\in U^c\textrm{ is a base for }G)>1-\frac{1}{m^{c-2}}
   \textrm{ for any }c\geq 3.\]
  Hence three random vectors form a base for $G$ with high probability if 
  $m$ is large. 
\end{thm}
\begin{proof}
  The $\FF qG$-module $V^c$ can be naturally identified with
  $M^{m\times c}(q)$, the space of $m\times c$-matrices over $\FF
  q$. Under this identification, $G$ acts on $M^{m\times c}(q)$ by
  permuting the rows of each element of $M^{m\times c}(q)$ in a
  natural way. Hence, a matrix $a\in M^{m\times c}(q)$ is a base for
  $G$ if and only if the rows of $a$ are pairwise different elements
  of $M^{1\times c}(q)$, the space of $c$-dimensional row vectors over
  $\FF q$. Thus, the probability in question is equal to the
  probability that $m$ random elements of $M^{1\times c}(q)$
  are pairwise different, which is 
  \[
  \prod_{i=0}^{m-1}\frac{q^c-i}{q^c}>\left(\frac{q^c-q}{q^c}\right)^m
  \geq \left(1-\frac{1}{m^{c-1}}\right)^m\geq 1-\frac{1}{m^{c-2}},
  \]
  where the first and second inequalities follows since $m<q$ by the
  coprime assumption.  The claim follows.
\end{proof}
\begin{cor}\label{cor:sym_deleted_perm}
  Let $V$ be an $n$ dimensional vector space over the finite field $\FF q$ and
  let $G=Z\cdot G_0\leq GL(V)$ be a coprime linear group, where
  $G_0\simeq S_m$ or $G_0\simeq A_m$ and $V$ is the non-trivial
  irreducible component of the natural $\FF q G_0$-module.
  Then we have 
  \[
  P(\textrm{random }\ul v\in V^c\textrm{ is a base for
  }G)>1-\frac{1}{n^{c-2}} \textrm{ for any }c\geq 3.
  \]
\end{cor}
\begin{proof}
  First, note that $\Fix(g)=0$ for every $g\in G\setminus G_0$, so a
  $\ul v\in V^c$ is a base for $G$ if and only if it is a base for
  $G_0$.  Second, let $U=V\oplus U_0$, where $U_0$ is the trivial
  module for $G_0$.  For any random vectors $u_1,\ldots,u_c\in U$ let
  $v_i$ be the projection of $u_i$ to $V$ along $U_0$. Then
  $u_1,\ldots,u_c$ is a base for $G_0$ if and only if $v_1,\ldots,v_c$
  is a base for $G_0$, so the claim follows from Theorem \ref{thm:sym_perm}.
\end{proof}  

\section{Proof of Theorem \ref{thm:randombase}}
Let $V$ be an $n$-dimensional vector space over the finite field $\FF
q$ and let $G\leq GL(V)=GL(n,q)$ be a coprime primitive linear group,
which is maximal, i.e. there is no coprime subgroup $L\leq GL(V)$
strictly containing $G$.  In the following, we give a structure
theorem of such groups very similar to a result about maximal solvable
primitive linear group (see \cite[Lemma 2.2]{Seress} and \cite[\S\S
19--20]{Suprunenko}).  Our proof uses ideas similar to those can be
found in \cite{Giudici}, \cite{HP}, and \cite{Suprunenko}. For the
convienience of the reader, we give a self-contained proof here.

In the following, we extend the vector space
structure on $V$ by defining multiplication on $V$ with elements from
a (possibly) larger field $\FF{q^k}\geq \FF q$ for some $k\mid n$. In
that way, $V$ will be both an $\FF q$-vector space and an
$\FF{q^k}$-vector space at the same time.

We will use the notation $V=V_n(q),\ V=V_d(q^k)$ or $V=V(q^k)$ 
if we would like to highlight the base field and/or the dimension of $V$. 

\begin{thm}\label{thm:generalcase}
  Let $V=V_n(q)$ be an $n$-dimensional vector space over the finite field
  $\FF q$ and let $G\leq GL(V)$ be a maximal coprime primitive linear group. 
  Then the following statements hold.
  \begin{enumerate}
  \item There is a unique maximal Abelian subgroup $Z\leq GL(V)$,
    which is normalised by $G$. Moreover, $Z$ is contained in $G$. 
  \item $Z$ is cyclic and $Z\cup\{0\}\simeq \FF{q^k}$ for some $k\mid n$.
  \item There is a (unique and maximal) $\FF{q^k}$ vector space
    structure $V=V_d(q^k)$ on $V$ for $d=n/k$ such that 
    $G\leq \Gamma L(d,q^k))$.
  \item Let $H:=G\cap GL(d,q^k)$. Then $Z\leq H=C_G(Z)\nor G$,
    furthermore $Z=Z(GL(d,q^k))$ is the group of scalar
    transformations on $V_d(q^k)$ and $G/H$ is included into the
    Galois group $\Gal(\FF{q^k},\FF q)$.
  \item Let $N=F^*(H)$ be the generalised Fitting subgroup of $H$. 
    Then $N/Z$ is the socle of $H/Z$. Furthermore, 
    $V_d(q^k)$ is an absolutely irreducible $\FF{q^k}N$-module. 
  \item Let $N_1,\ldots,N_t$ be the set of minimal normal subgroups
    of $H$ above $Z$. Then there is an absolutely irreducible $\FF{q^k}
    N_i$-module $V_i$ for every $i$ such that $V\simeq
    V_1\otimes_{\FF{q^k}}\ldots\otimes_{\FF{q^k}} V_t$.  Furthermore,
    $N=N_1\otimes N_2\otimes \ldots\otimes N_t$ and $H=H_1\otimes
    H_2\otimes\ldots\otimes H_t$ where $N_i\nor H_i\leq GL(V_i(q^k))$
    for every $i$.
  \item If $N_i/Z$ is Abelian, then $N_i=ZR_i$ where $R_i\leq N_i$ is
    an extraspecial $r_i$-group for some prime $r_i$ of order
    $r_i^{2l_i+1}$.  Furthermore, $|N_i/Z|=r_i^{2l_i}$ and
    $\dim_{\FF{q^k}}(V_i)=r_i^{l_i}$. 
  \item If $N_i/Z$ is a direct product of $s$ many isomorphic
    non-Abelian simple groups, then there is a tensor product
    decomposition $V_i=W_1\otimes \ldots\otimes W_s$ preserved by
    $N_i$.  Then $N_i=K_1\otimes \ldots\otimes K_s$ where $K_i=S_iZ$
    for each $i$, and the $S_i\leq GL(W_i)$ are isomorphic quasisimple
    absolutely irreducible groups. Finally, $H_i$ permutes the $K_i$-s
    and the $W_i$-s in a transitive way.
  \end{enumerate}
\end{thm}
\begin{proof}
  Let $A\leq GL(V)$ be any Abelian subgroup normalised by $G$ and $P$
  is the (unique) Sylow-$p$ subgroup of $A$ for
  $p=\operatorname{char}(\FF q)$. Then $P$ is normalised by $G$. Then
  $0\neq \Fix_V(P)=\cap_{p\in P}\Fix_V(p)\leq V$ is
  $G$-invariant. Since $V$ is an irreducible $\FF qG$-module, we get
  that $P=1$, so $|A|$ is coprime to $|V|$. Therefore, $GA\geq G$ is a
  coprime linear group, so $A\leq G$ by the maximality of $G$ and part
  of (1) is proved.

  Let $Z\leq GL(V)$ be a maximal Abelian subgroup normalised by
  $G$. By the previous paragraph, $Z\nor G$. Since $G\leq GL(V)$ is
  primitive linear, $V$ is a homogeneous $\FF q Z$-module. If
  $V=V_1\oplus\ldots\oplus V_d$ is a decomposition of $V$ into
  (isomorphic) irreducible $\FF q Z$-modules, then $Z\simeq
  Z_{V_i}\leq \End_Z(V_i)\simeq \FF {q^k}$ for some $k\geq 1$ by using
  Schur Lemma. Then $\langle Z\rangle_{\FF q}$ (the subalgebra of
  $\End(V)$ generated by $Z$) is isomorphic to the field $\FF q^k$,
  and it is invariant under the conjugation by elements of $G$. It
  follows that $\langle Z\rangle_{\FF q}\setminus\{0\}\simeq \FF q^*$
  is an Abelian subgroup of $GL(V)$ normalised by $G$. Therefore, (2)
  follows by the maximality of $Z$.
  
  Identifying $Z\cup 0\leq \End(V)$ with $\FF q^k$, it
  defines an $\FF {q^k}$ vector space structure on $V$.
  The conjugation action of $G$ on $Z\cup\{0\}=\FF {q^k}$ defines a
  homomorphism $\sigma:G\mapsto \Gal(\FF{q^k},\FF q)$. Now, for any
  $g\in G,\ \alpha\in \FF q^k$ and $v\in V$ we have $g(\alpha
  v)=(g\alpha q^{-1})g(v)=\alpha^{\sigma(g)}(v)$, so $G$ is included
  into the semilinear group $\Gamma L(V_d(q^k))=\Gamma L(d,q^k)$.  The
  subgroup $H$ is just the kernel of $\sigma$, so (4) and part
  of (3) follows.

  Let $B\nor G$ be any Abelian normal subgroup, $\alpha\in Z$ a
  generator of $Z$ and $b\in B$. Then $b\alpha
  b^{-1}=\alpha^{\sigma(b)}=\alpha^{q^s}$ for some $0\leq s<k$, so
  $[b,\alpha]=\alpha^{q^s-1}\in B$ is centralised by $b$.  Changing
  $b$ to $b^{-1}$ if necessary, we can assume that $0\leq s\leq k/2$.
  This means $(\alpha^{q^s-1})^{q^s}=\alpha^{q^s-1}$, so $q^k-1\mid
  (q^s-1)^2<q^k-1$. Therefore, $s=0$. Thus, $B\leq C_G(Z)$, so $BZ\geq
  Z$ is an Abelian normal subgroup in $G$. By the maximality of $Z$, we get 
  $B\leq Z$, which completes the proof of both (1) and (3).

  Let $M=F(H)$ be the Fitting subgroup of $H$. Then $Z(M)$ is an
  Abelian normal subgroup of $G$, so $Z(M)=Z$ by the maximality of
  $Z$.  Let $n$ by the nilpotency class of $M$. If $n=1$ then $M=Z$.
  Otherwise, we claim that $n=2$.  Assuming that $n\geq 3$, we have
  $1\neq \gamma_{n}(M)\leq Z$, and $[\gamma_{n-1}(M),\gamma_{n-1}(M)]\leq
  [\gamma_2(M),\gamma_{n-1}(M)]\leq \gamma_{n+1}(M)=1$, so $\gamma_{n-1}(M)$
  is an Abelian normal subgroup of $G$, so it must contained in $Z$.
  This forces $\gamma_n(M)=1$, a contradiction.  Therefore, $n\leq 2$,
  that is, $M/Z$ is Abelian. 

  Let $R$ be a Sylow-$r$-subgroup of $M$ for some prime $r$ dividing
  $|M/Z|$.  The commutator map defines a symplectic bilinear function
  from $R/Z$ into $Z(R)=R\cap Z$. Therefore, for any $x,y\in R$ we
  have $[x^r,y^r]=[x,y]^{r^2}=[x^{r^2},y]$.  If $r^s$ is the exponent
  of $R/(R\cap Z)$ for some $s\geq 2$, then $R^{r^{s-1}}Z$ is an
  Abelian normal subgroup of $G$, so $R^{r^{s-1}}\leq Z$, a
  contradiction. Thus, we get $R/(R\cap Z)$ is an elementary Abelian
  $r$-group.  Using this and the above commutator identity it also
  follows that $R'\leq Z$ is of exponent $r$. It follows that
  $R=(R\cap Z)R_0$ for some extraspecial $r$-group $R_0$.

  Be the previous two paragraphs, $F(H)/Z$ is exactly the direct
  product of the minimal Abelian normal subgroups of $H$, so $F(H)/Z$
  is contained in $\soc(H/Z)$. Since $N=F^*(H)$ is the central product
  of $F(H)$ and the layer $E(H)$, where $E(H)/Z$ is the direct product
  of the minimal non-Abelian normal subgroups of $H/Z$ it follows that
  $N/Z=\soc(H/Z)$ as claimed.  By \cite[Lemma 12.1]{Giudici},
  $V_d(q^k)$ is an absolutely irreducible $\FF {q^k}H$-module.  If the
  irreducible $\FF {q^k}N$-components of $V_d(q^k)$ were not be
  absolutely irreducible, then $Z(C_{GL(V_d(q^k))}(N))$ would be the
  multiplicative group of a proper field extension of $\FF {q^k}$
  normalised by $G$, which again contradicts with the maximility of
  $Z$. Now, let us assume that $V_d(q^k)=U\oplus\ldots\oplus U$ is a
  direct sum of $s$ many isomorphic absolutely irreducible $\FF
  {q^k}N$-modules for some $s\geq 2$. By \cite[Lemma
  4.4.3(ii)]{Kleidman}, there is a tensor product decomposition
  $U\otimes_{\FF{q^k}}W$ of $V_d(q^k)$ such that $N\leq GL(U)\otimes
  1_W\leq GL(U)\otimes GL(W)$ and $G\leq N_{\Gamma L(V)}(N)\leq
  N_{\Gamma L(V)}(GL(U)\otimes GL(W))$. Let $L=\{1_U\otimes
  h_W\,|\,\exists h_U\in GL(U)\textrm{ such that }h_U\otimes h_W\in
  H\}$. If $L=Z$, then $V_d(q^k)$ is not irreducible as an $\FF
  {q^k}H$-module, a contradiction.  We have $L\leq GL(V)$ is a coprime
  linear group normalised by $G$, so $LG\leq GL(V)$ is coprime. Using
  the maximality of $G$ we get that $L<G$. But then $Z< L\leq H$
  clearly centralises $N=F^*(H)$, a contradiction. So, $V_d(q^k)$ is
  an absolutely irreducible $\FF{q^k}N$-module, and (5) is proved.
  Now, (6) follows by a combined use of \cite[Corollary
  18.2/(a)]{MalleTesterman} and \cite[Lemma 4.4.3(iii)]{Kleidman}.

  If $N_i/Z$ is Abelian, then it is a minimal Abelian normal subgroup
  of $H/Z$ so it is elementary Abelian $r_i$-group for some prime
  $r_i$. Using the same argument as in paragraph 6 of this proof, one
  can find the extraspecial $r_i$-group $R_i$ by taking the full
  inverse image of a maximal non-degenerate subspace of $R/R'$ where
  $R$ is the Sylow-$r_i$ subgroup of $N_i$. For this subgroup, it
  clearly follows that $N_i=ZR_i$, and $|R_i|=r_i^{2l_i+1}$ for some
  integer. Furthermore, since $V_i$ is an absolutely irreducible
  $\FF{q^k}N_i$-module, it must be an absolutely irreducible
  $\FF{q^k}R_i$-module. It is well-known that extraspecial $r_i$-group of order 
  $r_i^{2l_i+1}$ has a unique faithfull absolutely irreducible 
  ordinary representation, and this representation has degree $r_i^{l_i}$, 
  which finishes the proof of (7). 

  Finally, (8) can again be deduced from \cite[Corollary
  18.2/(a)]{MalleTesterman} and from the fact that $N_i/Z$ is a
  minimal normal subgroup in $H_i/Z$.
\end{proof}
\begin{lem}\label{lem:commutator_supp}
  Let $G$ be a group, $K$ be a field and 
  let $V$ be an arbitrary finite dimensional $KG$-module.
  \begin{enumerate}
  \item For any $g,h\in G$ we have $\supp([g,h])\leq 2\supp(g)$.
  \item If $N\nor G$ such that $V$ is absolutely irreducible as a
    $KN$-module, then $\msupp(N)\leq 2\msupp(G)$.
  \end{enumerate}
\end{lem}
\begin{proof}
  Let us consider the subspaces $U=\Fix(g)$ and $W=\Fix(h^{-1}gh)$ of $V$. 
  Then we have 
  \[\dim (U)+\dim (W)-\dim(U\cap W)=\dim(U+W)\leq \dim (V).\]
  Using that $\dim (U)=\dim (W)=\dim (V)-\supp(g)$ we get 
  $\dim(U\cap W)\geq \dim (V)-2\supp(g)$. 
  On the other hand $U\cap W\leq \Fix([g,h])$ holds trivially,
  so 
  \begin{align*}
    \supp([g,h])&=\dim (V)-\dim(\Fix([g,h]))\leq \dim (V)-\dim(U\cap W)\\
    &\leq \dim(V)-(\dim (V)-2\supp(g))=2\supp(g),
  \end{align*}               
  and part (1) follows.  

  For part (2), let $1\neq g\in G$ be any element. If $[g,N]=1$, then
  $g$ acts as a scalar transformation on $V$ by \cite[Theorem
  9.2]{Ibook}, so $\supp(g)=\dim(V)\geq \msupp(N)$. Otherwise, there
  is an element $n\in N$ such that $[g,n]\neq 1$. Then we have $\msupp(N)\leq
  \supp([g,n])\leq 2\supp(g)$. Thus, $\msupp(N)\leq 2\supp(g)$ for
  every $1\neq g\in G$, which proves that $\msupp(N)\leq 2\msupp(G)$.
\end{proof}
\begin{lem}\label{lem:tensor_supp}
  Let $V_1,\ldots,V_k$ be finite dimensional vector spaces over the
  field $\FF q$ and $Z< G_1\leq GL(V_1),\ldots,Z< G_k\leq GL(V_k)$ be
  coprime linear groups.  Consider the group
  $G:=G_1\otimes\ldots\otimes G_k$ acting on the tensor product
  $V:=V_1\otimes\ldots\otimes V_k$ in a natural way.
  \begin{enumerate}
  \item Let $g=g_1\otimes\ldots\otimes g_k\in G$ with $g_j\in G_j$
    for each $j$ and let us assume that $g_i\notin Z$ for some $i$.
    Then 
    \[\supp_V(g)\geq \msupp_{V_i}(G_i)\cdot \frac{\dim(V)}{\dim(V_i)}\]
    or $\supp_V(g)\geq \frac{1}{2}\dim(V)$.
  \item As a consequence
    \[
    \msupp_V(G)=\min_i
    \Big\{\msupp_{V_i}(G_i)\cdot\frac{\dim(V)}{\dim(V_i)}\Big\},
    \]
    or $\msupp_V(G)\geq \frac{1}{2}\dim(V)$.
  \end{enumerate}
\end{lem}
\begin{proof}
  To prove part (1), first we consider the case $k=2$. Let $n_1=\dim
  (V_1),\,n_2=\dim (V_2)$, so $n=\dim (V)=n_1n_2$. Furthermore, let $1\neq
  g=g_1\otimes g_2\in G_1\otimes G_2$ be an element of $G$ with
  $g_1\notin Z$.  Since the action is coprime, $g_1$ and $g_2$ are
  diagonalisable over $\CF q$. Let $\alpha_1,\ldots\alpha_s\in\CF q$
  be the different eigenvalues of $g_1$ with multiplicity
  $k_1,k_2,\ldots,k_s$. We can assume that $k_1$ is the largest among
  the $k_i$. Let $l_1,\ldots,l_s$ be the multiplicities of
  $\alpha_1^{-1},\ldots,\alpha_s^{-1}$ in the characteristic
  polynomial of $g_2$ (Some of them can be zero).  Then
  \begin{align*}
  \supp_V(g)&=\supp_V(g_1\otimes g_2)=n-\dim(\Fix_V(g_1\otimes g_2))\\
  &=n-\sum_{i=1}^s k_il_i\geq 
  n-\sum_{i=1}^s k_1l_i\geq (n_1-k_1)n_2.
  \end{align*}
  If $\alpha_1\in \FF q$, then we can substitute $g_1$ by
  $\alpha_1^{-1}g_1$ and $g_2$ by $\alpha_1 g_2$ (since both $G_1$ and
  $G_2$ contains all the scalar transformations), so we can assume
  that $\alpha_1=1$. Now, since $g_1\neq 1$, we get $\supp_V(g)\geq
  (n_1-k_1)n_2 = \supp_{V_1}(g_1)n_2 \geq \msupp_{V_1}(G_1)\cdot
  \dim(V_2)$. 

  Now, let us assume that $\alpha_1\notin \FF q$. Then there is an
  algebraic conjugate element of $\alpha_1$ (different from
  $\alpha_1$) under the action of $\textrm{Gal}(\CF q,\FF q)$ which is
  also an eigenvalue of $g_1$ with the same multiplicty as $\alpha_1$.
  In particular, $k_1\leq n_1/2$.
  Thus, \[\supp_V(g)\geq (n_1-k_1)n_2\geq (n_1/2)n_2=\frac{\dim(V)}2.\]
  By changing the role of $g_1$ and $g_2$ in the proof and
  by using induction on $k$, we get the claim of part (1).

  Finally, if
  $\supp_{V_i}(g_i)=\msupp_{V_i}(G_i)$ for some $g_i\notin Z$, then
  \[\supp_{V}(1\otimes\ldots \otimes 1\otimes g_i\otimes 1\otimes\ldots 
  \otimes 1)=\msupp_{V_i}(G_i)\cdot\frac{\dim(V)}{\dim(V_i)},
  \]
  so part (2) follows by using part (1).
\end{proof}
\begin{proof}[Proof of Theorem \ref{thm:randombase}]
  Let $G\leq GL(V)$ be a coprime primitive linear group.
  Without loss of generality we can
  assume that $G$ is maximal among such subgroups of $GL(V)$.
  Let $Z$ be the
  unique maximal Abelian subgroup in $GL(V)$ which is normalised by
  $G$ and $H$ be the intersection of $G$ and $GL(d,q^k)$ as in Theorem
  \ref{thm:generalcase}.
  If $g\in G\setminus H$ then there is a $z\in
  Z$ such that $[g,z]\neq 1$. 
  By Lemma \ref{lem:commutator_supp},
  $\supp_V(g) \geq \frac{1}{2}\supp_V([g,z])=\frac{1}{2}\dim(V)$.
  Therefore if $c>4$, then
  \[\sum_{g\in G\setminus H}\frac{1}{q^{c \cdot \supp_V(g)}}\leq
  \frac{|G\setminus H|}{q^{\frac{c}{2} \dim(V)}}
  \leq \frac{|V|^2}{|V|^{\frac{c}{2}}}\leq \frac{1}{|V|^{\frac{c}{2}-2}}
  .\]

  Now let $g$ be an element of $H=H_1\otimes \ldots\otimes H_t$.  So
  $g=(g_1,\ldots,g_t)$ where $g_i\in H_i$ for all $i\in[t]$ and $g$
  preserves the tensor product decomposition $V=
    V_1\otimes\ldots\otimes V_t$ over $\FF q^k$ as in Theorem
  \ref{thm:generalcase} (6) and $\dim_{\mathbb{F}_{q^k}}(V_i)=d_i$ for
  all $i$ (therefore $d=\dim_{\FF{q^k}}(V)=\prod_{i=1}^t d_i$ and
  $\dim(V)=\dim_{\FF q}(V)=k\cdot \prod_{i=1}^t d_i$).  We can assume that in
  this decomposition the dimensions of the vector spaces are
  decreasing, i.e. $d_1\geq d_2\geq \ldots \geq d_t \geq 2$.  Let
  $g_i\notin Z$ for some $i\neq 1$.  Then by Lemma
  \ref{lem:tensor_supp}, 
  \[
  \supp_{V}(g)\geq \msupp_{V_i}(H_i)\cdot
  \frac{\dim(V)}{\dim(V_i)}\geq k \prod_{j\neq i}d_j.
  \] 
  Since 
  $2\prod_{j\neq i}d_j\geq 2^{t-1}d_1\geq \sum_{i=1}^t d_1\textrm{ and } 
  k\prod_{j\neq i}d_j\geq k\sqrt{\prod_{j=1}^t d_j}\geq \sqrt{\dim(V)}$
  we get that
  \[ 
  c\cdot \supp_{V}(g) \geq 2k \sum_{i=1}^t d_i + (c-4) \sqrt{\dim(V)}.
  \]
  Hence,
  \begin{align*}
  \sum_{g\in H\setminus H_1}\frac{1}{q^{c\cdot \supp_V(g)}}&
  \leq \frac{\prod_{i=1}^t|H_i|}{q^{2k \sum_{i=1}^t d_i + (c-4) \sqrt{\dim(V)}}}\\
  &\leq \frac{q^{2k \sum_{i=1}^t d_i}}{q^{2k \sum_{i=1}^t d_i + (c-4)
      \sqrt{\dim(V)}}}=\frac{1}{q^{(c-4)\sqrt{\dim(V)}}}.
  \end{align*}
  
  Now assume that $g\in H_1$.  In this case
  $\supp_V(g)=\supp_{V_1}(g)\cdot \frac{d}{d_1}$.  By
  Theorem \ref{thm:generalcase}, $Z\leq N_1\leq H_1\leq GL(V_1(q^k))$
  where $N_1$ is a minimal normal subgroup above $Z$ and $N_1/Z$ is
  characteristically simple.  Therefore, it is either an
  elementary Abelian group,  a direct product of non-Abelian simple groups,
  or a non-Abelian simple group. 
  
  First, if $N_1/Z$ is elementary Abelian, then $N_1=Z\cdot P$ where $P$ is
  an extraspecial $r$-group for a prime $r$ with $r\mid q^k-1$.  Then
  $V_1(q^k)$ is an absolutely irreducible $\FF{q^k}P$-module.  If
  $n\in P\setminus Z$ then $n$ has exactly $r$ different eigenvalues
  on $V_1$ (or on $V$) each with the same multiplicity.  
  It follows that $\msupp_V(N_1)\geq \frac{r-1}{r}\dim(V)\geq
  \frac{1}{2}\dim(V)$, so $\msupp_{V}(H_1)\geq \frac{1}{4}\dim(V)$ by Lemma 
  \ref{lem:commutator_supp}.  In this case,
  \[\sum_{g\in H_1} \frac{1}{q^{c\cdot \supp_V(g)}}
  \leq \frac{|H_1|}{q^{\frac{c}{4} \dim(V)}}
  \leq \frac{|V|^2}{q^{2\dim(V)+(\frac{c}4-2)\dim(V)}}
  \leq \frac{1}{|V|^{\frac{c}4-2}}.\]
  
  Next, let $N_1/Z$ is a direct product of $s\geq 2$ many isomorphic
  non-Abelian simple groups. By Theorem \ref{thm:generalcase} (8), the
  action of $N_1=K_1\otimes\ldots\otimes K_s$ on $V_1$ preserves a
  tensor product decomposition $V_1=W_1\otimes \ldots \otimes W_s$
  over $\FF q^k$ , where $\dim_{\FF{q^k}}(W_i)=\sqrt[s]{d_1}\geq 2$ for
  every $i$.  Using \cite[Theorem 1]{P3-Pyber}, we get that 
  \[
  |N_1|\leq \prod_{i=1}^s|K_s|\leq \prod_{i=1}^s |W_i|^2=q^{2ks\sqrt[s]{d_1}}.
  \]
  On the other hand, $H_1/N_1$ acts faithfully on $\{W_1,\ldots,W_s\}$ and 
  $|H_1/N_1|$ is coprime to $q$, so $|H_1/N_1|\leq q^s$ by 
  \cite[Corollary 2.4]{HP}. Therefore, $|H_1|\leq q^{2ks\sqrt[s]{d_1}+s}$.
  By Lemma \ref{lem:commutator_supp} and by Lemma \ref{lem:tensor_supp},
  \[
  \supp_{V_1}(g) \geq \frac{1}{2}\msupp_{V_1}(N_1)\geq \frac{k}{2}\cdot
  \frac{d_1}{\sqrt[s]{d_1}}.
  \]  
  Therefore, 
  \[
  c\supp_V(g)\geq
  5k{d_1}^{(s-1)/s}+\Big(\frac{c}2-5\Big)k\sqrt{d_1}\cdot\frac{d}{d_1}
  \geq 2ks\sqrt[s]{d_1}+s+\Big(\frac{c}2-5\Big)\sqrt{\dim(V)}.
  \]
  So,
  \begin{align*}
    \sum_{1\neq g\in H_1}\frac{1}{q^{c\cdot\supp_V(g)}}
    &\leq \frac{|H_1|}{q^{c\cdot\msupp_V(H_1)}}
    \leq \frac{q^{2ks\sqrt[s]d_1+s}}{q^{2ks\sqrt[s]{d_1}+s+
        (\frac{c}2-5)\sqrt{\dim(V)}}}\\
    &\leq \frac{1}{q^{(\frac{c}{2}-5)\sqrt{\dim(V)}}}.
  \end{align*}
  Finally, let $N_1/Z$ be a non-Abelian simple group.  If $d_1\leq
  \sqrt{d}$, then we can use the same argument as in the previous paragraph
  to get that
  \[
  \sum_{1\neq g\in H_1}\frac{1}{q^{c\cdot\supp_V(g)}}\leq 
  \frac{1}{q^{(c-2)\sqrt{\dim(V)}}}.
  \]
  Summarizing the bounds given until this point, we get 
  that 
  \begin{align*}
  Pb(c,G,V)&\geq 1-\sum_{1\neq g\in G}\frac{1}{q^{c\cdot\supp_V(g)}}
    \geq 1-\Big(\frac{1}{|V|^{\frac{c}2-2}}+\frac{1}{q^{(c-4)\sqrt{\dim(V)}}}\\
    &+\frac{1}{q^{(\frac{c}2-5)\sqrt{\dim(V)}}}\Big)
    \geq 1-\frac{3}{q^{(\frac{c}2-5)\sqrt{\dim(V)}}},
  \end{align*}
  which is case (1) of Theorem \ref{thm:randombase}.

  Now, let us assume that $d_1\geq \sqrt{d}$. If $|V_1|=q^{kd_1}$ is
  bounded by the constant appearing in part (3) of Theorem
  \ref{thm:quasisimple}, then $|V|$ is also bounded. Hence we can
  assume that either part (1) or part (2) of Theorem
  \ref{thm:quasisimple} holds.  By Lemma \ref{lem:commutator_supp}, we
  also have $\msupp_V(H_1)\geq \frac{1}2\msupp_V(N_1)$.  

  If $N_1/Z$ is not
  an alternating group, then $\msupp_V(N_1)\geq \frac{1}{40}\dim(V)$
  and $5\cdot\msupp_{V}(N_1)\geq \log_q|H_1|$ by using Corollary
  \ref{cor:msupp-dim_bound}, Theorem \ref{thm:quasisimple}/(1) and
  Lemma \ref{lem:tensor_supp}/(2).
  Thus, we have 
  \[
  \sum_{1\neq g\in H_1}\frac{1}{q^{c\cdot\supp_V(g)}}\leq 
  \frac{1}{|V|^{(c-10)/80}}.
  \]
  So, in this case we get that
  \begin{align*}
    Pb(c,G,V)&
    \geq 1-\Big(\frac{1}{|V|^{\frac{c}2-2}}+\frac{1}{q^{(c-4)\sqrt{\dim(V)}}}+
    \frac{1}{|V|^{(c-10)/80}}\Big)\\
    &\geq 1-\Big(\frac{1}{q^{(c-4)\sqrt{\dim(V)}}}+
    \frac{2}{|V|^{(c-10)/80}}\Big),
  \end{align*}
  which is case (2)/a of Theorem \ref{thm:randombase}.

  Finally, let $N_1/Z\simeq A_m$ for some $m$.  If $V_1$ is not an
  irreducible component of the natural $\FF q^k A_m$ permutation
  module, then we have $\msupp_{V}(N_1)\geq
  \frac{1}{16}\sqrt{\dim(V)}$ and $5\cdot\msupp_{V}(N_1)\geq
  \log_q|H_1|$ by using Corollary \ref{cor:msupp-dim_symmetric},
  Theorem \ref{thm:quasisimple}/(1) and Lemma
  \ref{lem:tensor_supp}/(2).
  Thus, we have 
  \[
  \sum_{1\neq g\in H_1}\frac{1}{q^{c\cdot\supp_V(g)}}\leq 
  \frac{1}{q^{\frac{c-10}{16}\sqrt{\dim(V)}}}
  \]
  and 
  \begin{align*}
    Pb(c,G,V)&
    \geq 1-\Big(\frac{1}{|V|^{\frac{c}2-2}}+\frac{1}{q^{(c-4)\sqrt{\dim(V)}}}+
    \frac{1}{q^{\frac{c-10}{16}\sqrt{\dim(V)}}}\Big)\\
    &\geq 1-\frac{3}{q^{\frac{c-10}{16}\sqrt{\dim(V)}}}.
  \end{align*}
  Finally, if $V_1$ is the non-trivial irreducible component of the
  natural $\FF q^kA_m$-module, then with the use of Corollary
  \ref{cor:sym_deleted_perm} we get that
  \[
    Pb(c,G,V)\geq
    1-\Big(\frac{1}{|V|^{\frac{c}2-2}}+\frac{1}{q^{(c-4)\sqrt{\dim(V)}}}+
    1-Pb(c,H_1,V)\Big)\geq 1-\frac{3}{n^{c-2}},
  \]
  which completes the proof of Theorem \ref{thm:randombase}.
\end{proof}

\end{document}